\documentclass{amsart}
\usepackage{amssymb,latexsym,amsmath,amscd,graphicx,setspace,amsthm,verbatim,comment}
\usepackage{tikz,tkz-berge}
\usepackage[margin = 3 cm]{geometry} \input xy \xyoption{all}

\usetikzlibrary{shapes.geometric}

\tikzset{every loop/.style={min distance=10mm,looseness=10}}

\newcount\mycount

\theoremstyle{plain}
\newtheorem{theorem}{Theorem}[section]
\newtheorem{proposition}[theorem]{Proposition}

\newtheorem{corollary}[theorem]{Corollary}
\newtheorem{lemma}[theorem]{Lemma}

\makeatletter
\def\th@remark{%
  \thm@headfont{\bfseries}%
  \normalfont 
  \thm@preskip\topsep \divide\thm@preskip\tw@
  \thm@postskip\thm@preskip
}
\makeatother

\theoremstyle{remark}

\theoremstyle{definition}
\newtheorem{definition}[theorem]{Definition}

\begin{document} 
\title[On abelian $\ell$-towers of multigraphs]{On abelian $\ell$-towers of multigraphs}

\author{Daniel Valli\`{e}res}

\address{Mathematics and Statistics Department, California State University, Chico, CA 95929 USA}
\email{dvallieres@csuchico.edu}

\date{\today}

\begin{abstract}
We study how the $\ell$-adic valuation of the number of spanning trees varies in regular abelian $\ell$-towers of multigraphs.  We show that for an infinite family of regular abelian $\ell$-towers of bouquets, the behavior of the $\ell$-adic valuation of the number of spanning trees behave similarly to the $\ell$-adic valuation of the class numbers in $\mathbb{Z}_{\ell}$-extensions of number fields.
\end{abstract} 
\maketitle 
\tableofcontents 

\section{Introduction}
There is a host of conjectures in the field of special values of $L$-functions in algebraic number theory, and it might be interesting to see if there are analogous conjectures and theorems for multigraphs.  In \cite{Hammer:2020}, the authors studied a few such results regarding the special value $u=1$ of Artin-Ihara $L$-functions associated to abelian covers of multigraphs.  They showed a result analogous to the classical Brumer conjecture, and they also computed the index of an ideal analogous to the Stickelberger ideal in algebraic number theory.

Let $K$ be an algebraic function field over a finite field $\mathbb{F}_{q}$, where $q$ is a power of a prime, and fix an arbitrary rational prime $\ell$.  For $n \ge 0$, let $K_{n} = K\mathbb{F}_{\ell^{n}}$, where $\mathbb{F}_{\ell^{n}}$ is the unique subfield of $\overline{\mathbb{F}}_{q}$ such that $[\mathbb{F}_{\ell^{n}}:\mathbb{F}_{q}] =  \ell^{n}$.  Then, we have a tower of fields
$$K = K_{0} \subseteq K_{1} \subseteq K_{2} \subseteq \ldots \subseteq K_{n} \subseteq \ldots $$
Let $h_{n}$ be the cardinality of ${\rm Pic}^{\circ}(K_{n})$ for $n \ge 1$.
\begin{theorem} \label{function_field}
There exist constants $\lambda_{\ell}$ and $\nu_{\ell}$, and an integer $n_{0} \ge 1$ such that 
$${\rm ord}_{\ell}(h_{n}) = \lambda_{\ell} \cdot n + \nu_{\ell}, $$
for all $n \ge n_{0}$.
\end{theorem}
\begin{proof}
See \S $11$ of \cite{Rosen:2002} for instance.
\end{proof}
In \cite{Iwasawa:1973}, Iwasawa proved an analogous result for $\mathbb{Z}_{\ell}$-extensions of number fields.  Let $K$ be a number field, and suppose we have a sequence of fields
$$K = K_{0} \subseteq K_{1} \subseteq K_{2} \subseteq \ldots \subseteq K_{n} \subseteq \ldots $$
forming a $\mathbb{Z}_{\ell}$-extension of $K$.  Now, we let $h_{n}$ be the cardinality of ${\rm Cl}(K_{n})$ for $n \ge 1$.  Then Iwasawa showed the following theorem.
\begin{theorem}[Iwasawa] \label{iwasawa}
There exist constants $\mu_{\ell}$, $\lambda_{\ell}$, $\nu_{\ell}$, and an integer $n_{0} \ge 1$ such that
$${\rm ord}_{\ell}(h_{n}) = \mu_{\ell} \cdot \ell^{n} + \lambda_{\ell} \cdot n + \nu_{\ell}, $$
for all $n \ge n_{0}$.
\end{theorem}

In this paper, we investigate the following question:  Is there a result analogous to Theorem \ref{function_field} or Theorem \ref{iwasawa} for (some) abelian $\ell$-towers of multigraphs, where the cardinality of ${\rm Pic}^{\circ}(K_{n})$ in the function field case or the cardinality of ${\rm Cl}(K_{n})$ in the number field case is replaced with the number of spanning trees in the multigraph case?

The paper is organized as follows.  In \S \ref{pre}, we gather together some well-known results about algebraic number theory and graph theory that will be used throughout the paper.  In \S \ref{number}, using Artin-Ihara $L$-functions, we derive a useful formula for the number of spanning trees of the top multigraph of an abelian cover of multigraphs.  Then, we move on in \S \ref{upper_bound_1} to find an upper bound for the number of spanning trees at each layer of a regular abelian $\ell$-tower of multigraphs.  In \S \ref{over_bouquet}, we specialize to the case where the bottom multigraph is a bouquet.  This simplifies the matters quite a bit, since among other things the number of spanning trees of a bouquet is one.  We show in \S \ref{lower_bound_1} a lower bound for the number of spanning trees at each layer of a regular abelian $\ell$-tower of a bouquet.  The Galois covers of bouquets are precisely the Cayley-Serre multigraphs, and we remind the reader about them in \S \ref{cayley_serre}.  It is well-known that the eigenvalues of the adjacency matrix of a Cayley-Serre multigraph can be found explicitly, and we explain this in \S \ref{spectra_cs} in a way that is useful for us.    In \S \ref{cyclo_integer}, we obtain a few simple results about some algebraic integers in cyclotomic fields of prime power conductors.  We also define some polynomials recursively which are used to prove our main result in \S \ref{main_1}.  We end the paper with a few examples in \S \ref{examples}.

\subsection*{Acknowledgement}
I would like to thank Thomas Mattman, Kevin McGown, and Jonathan Sands for various stimulating discussions.

\section{Preliminaries} \label{pre}
\subsection{Algebraic number theory}
All the facts in this section can be found in \cite{Neukirch:1999} for instance.  Recall that a discrete valuation $v$ on a field $k$ is a function
$$v:k \longrightarrow \mathbb{Z} \cup \{ \infty\} $$
satisfying 
\begin{enumerate}
\item $v(x) = \infty$ if and only if $x = 0$,
\item $v(x \cdot y) = v(x) + v(y)$ for all $x,y \in k$,
\item $v(x + y) \ge {\rm min}(v(x),v(y))$ for all $x,y \in k$.
\end{enumerate}
Furthermore, if $v(x) \neq v(y)$, then
\begin{equation} \label{equality}
v(x + y) = {\rm min}(v(x),v(y)).
\end{equation}
If $k$ is a number field, then any prime ideal $\mathfrak{p}$ of its ring of integers $O_{k}$ gives rise to a discrete valuation ${\rm ord}_{\mathfrak{p}}$ as follows.  If $\lambda \in k^{\times}$, and $(\lambda) = \mathfrak{p}^{a} \cdot \mathfrak{a}$, where $(\mathfrak{p},\mathfrak{a}) = 1$, then 
$${\rm ord}_{\mathfrak{p}}(\lambda)=a. $$

If $K/k$ is an extension of number fields and $\mathfrak{P}$ is a prime ideal of $O_{K}$ lying above the prime ideal $\mathfrak{p}$ of $O_{k}$ with ramification index $e = e(\mathfrak{P};\mathfrak{p})$, then
\begin{equation} \label{ramif}
e \cdot {\rm ord}_{\mathfrak{p}}(x) = {\rm ord}_{\mathfrak{P}}(x), 
\end{equation}
for all $x \in k$. 

Throughout this paper, for an integer $m \ge 1$, we let
$$\zeta_{m} = \exp\left( \frac{2 \pi i}{m}\right).$$
The cyclotomic number field $\mathbb{Q}(\zeta_{m})$ is a $CM$-field if $m \neq 1,2$, and we denote its maximal real subfield by $\mathbb{Q}(\zeta_{m})^{+}$ (also when $m = 1,2$, in which case $\mathbb{Q}(\zeta_{m})^{+} = \mathbb{Q}$).  In particular, if $\ell$ is a rational prime number and $i$ a positive integer, then there is a unique prime $\mathcal{L}_{i}$ lying above $\ell$ in $\mathbb{Q}(\zeta_{\ell^{i}})$ which is totally ramified with ramification index $e_{i} = \varphi(\ell^{i})$.  (The function $\varphi$ is the Euler phi function.)  Furthermore,
\begin{equation} \label{unique_prime}
\mathcal{L}_{i} = (1 - \zeta_{\ell^{i}}).
\end{equation}
The unique prime ideal of $\mathbb{Q}(\zeta_{\ell^{i}})^{+}$ lying between $\ell$ and $\mathcal{L}_{i}$ will be denoted by $\mathcal{L}_{i}^{+}$.

If $G$ is a finite abelian group, then we let $\widehat{G}$ denote the characters of $G$.  Given $\chi \in \widehat{G}$, we define the cyclotomic number field
$$\mathbb{Q}(\chi) = \mathbb{Q}(\{\chi(\sigma) \, | \, \sigma \in G \}), $$
and we denote its ring of integers by $\mathbb{Z}[\chi]$.  The ring of integers of $\mathbb{Q}(\chi)^{+}$ will be denoted by $\mathbb{Z}[\chi]^{+}$.

\subsection{Graph theory}
Our main references for this section are \cite{Corry/Perkinson:2018}, \cite{Sunada:2013}, and \cite{Terras:2011}.  We think of multigraphs as presented in \cite{Sunada:2013}.  Thus, a multigraph $X$ consists of a set of vertices $V_{X}$ and a set of directed edges $E_{X}^{+}$ with two functions
$$i:E_{X}^{+} \longrightarrow V_{X} \times V_{X} \text{ and } \iota:E_{X}^{+} \longrightarrow E_{X}^{+}$$
satisfying:
\begin{enumerate}
\item $\iota^{2} = {\rm id}_{E_{X}^{+}}$,
\item $\iota(e) \neq e$ for all $e \in E_{X}^{+}$,
\item $i(\iota(e)) = \tau(i(e))$ for all $e \in E_{X}^{+}$,
\end{enumerate}
where $\tau:V_{X} \times V_{X} \longrightarrow V_{X} \times V_{X}$ is defined via $\tau(x,y) = (y,x)$.  The function $i$ is called the incidence map, and the function $\iota$, the inversion map.  Writing
$$i(e) = (o(e),t(e)), $$
we obtain two functions $o,t:E_{X}^{+} \longrightarrow V_{X}$, called the origin and the terminus maps.  If $v \in V_{X}$, we let
$$E_{v} = \{e \in E_{X}^{+} \, | \, o(e) =v \}. $$
The quotient $E_{X} = E_{X}^{+}/\langle \iota \rangle$ consists of the set of undirected edges.  Note that $2|E_{X}| = |E_{X}^{+}|$.  A multigraph is said to be finite if both $V_{X}$ and $E_{X}$ are finite set.  The edges $e$ with $o(e) = t(e)$ are called loops, and the edges $e$ with $o(e) \neq t(e)$ will be called simple edges.    If $v \in V_{X}$, then the valency of $v$ is defined to be
$${\rm val}_{X}(v) = |E_{v}|,$$
provided $|E_{v}|$ is finite.

As usual, we define the Betti numbers of a multigraph $X$ to be
$$b_{0}(X) = {\rm rank}_{\mathbb{Z}} \, H_{0}(X,\mathbb{Z}) \text{ and } b_{1}(X) = {\rm rank}_{\mathbb{Z}} \, H_{1}(X,\mathbb{Z}).$$
The Euler characteristic of $X$ is then
$$\chi(X) = b_{0}(X) - b_{1}(X). $$
The $0$-th Betti number $b_{0}(X)$ is the number of connected components of $X$ and $b_{1}(X) = |E_{X}| - |V_{X}| + 1$ if $X$ is finite.  \emph{Throughout this paper, by a multigraph, we will always mean a finite connected multigraph without any degree one vertex, unless otherwise stated}  Thus, we will always have
$$\chi(X) = |V_{X}| - |E_{X}|.$$
We also let $g_{X} = |V_{X}|$ and we might drop the index $X$ if the graph is clear from the context.

The divisor group on $X$ is defined to be the free abelian group on the vertices $V_{X}$.  It is an abelian group denoted by ${\rm Div}(X)$.  If $D = \sum_{v \in V_{X}}n_{v} \cdot v \in {\rm Div}(X)$, then we define
$${\rm deg}(D) = \sum_{v \in V_{X}} n_{v}.$$
This gives a group morphism ${\rm deg}:{\rm Div}(X) \longrightarrow \mathbb{Z}$ whose kernel will be denoted by ${\rm Div}^{\circ}(X)$.  We let $\mathcal{M}(X) = \mathbb{Z}^{V_{X}}$ be the set of $\mathbb{Z}$-valued functions on $V_{X}$.  For $v \in V_{X}$, we define $\chi_{v} \in \mathcal{M}(X)$ via
\begin{equation*}
\chi_{v}(v_{0}) = 
\begin{cases}
1, &\text{ if } v = v_{0};\\
0, &\text{ if } v \neq v_{0}.
\end{cases}
\end{equation*}
The functions $\chi_{v}$, as $v$ runs over $V_{X}$, form a $\mathbb{Z}$-basis for $\mathcal{M}(X)$.  One then defines a group morphism ${\rm div}:\mathcal{M}(X) \longrightarrow {\rm Div}(X)$ on the basis elements $\chi_{v}$ (and extending by $\mathbb{Z}$-linearity) via 
$$\chi_{v_{0}} \mapsto {\rm div}(\chi_{v_{0}}) = \sum_{v \in V_{X}} \rho_{v}(v_{0}) \cdot v, $$
where
\begin{equation*}
\rho_{v}(v_{0}) = 
\begin{cases}
 {\rm val}_{X}(v_{0}) - 2\cdot \text{number of undirected loops at $v_{0}$}, &\text{ if } v = v_{0}; \\
-\text{number of undirected edges from $v$ to $v_{0}$}, &\text{ if } v \neq v_{0}.
\end{cases}
\end{equation*}
We let ${\rm Pr}(X) = {\rm div}(\mathcal{M}(X))$ which is a subgroup of ${\rm Div}(X)$ and furthermore, we let ${\rm Pic}(X) = {\rm Div}(X)/{\rm Pr}(X)$. Note that ${\rm Pr}(X) \subseteq {\rm Div}^{\circ}(X)$, and thus one also defines ${\rm Pic}^{\circ}(X) = {\rm Div}^{\circ}(X)/{\rm Pr}(X)$.  The group ${\rm Pic}^{\circ}(X)$ is always finite and we let $\kappa_{X} = |{\rm Pic}^{\circ}(X)|$.  It is well known that $\kappa_{X}$ is the number of spanning trees of $X$.  (Note that if $g_{X} = 1$, then $\kappa_{X} = 1$ and the unique spanning tree is the subgraph consisting of the unique vertex.)  For more details, we refer the reader to \cite{Corry/Perkinson:2018}.

We remind the reader about a few matrices attached to multigraphs.  From now on, we introduce a labeling of the vertices $V_{X} = \{v_{1},\ldots,v_{g_{X}} \}$.
\begin{definition}
Let $X$ be a multigraph.
\begin{enumerate}
\item The adjacency matrix $A$ attached to $X$ is the $g_{X} \times g_{X}$ matrix $A=(a_{ij})$ defined via
\begin{equation*}
a_{ij} = 
\begin{cases}
\text{Twice the number of undirected loops at the vertex }i, &\text{ if } i=j;\\
\text{The number of undirected edges connecting the $i$th vertex to the $j$th vertex}, &\text{ if } i \neq j.
\end{cases}
\end{equation*}
\item The valency matrix $D$ attached to $X$ is the $g_{X} \times g_{X}$ diagonal matrix $D = (d_{ij})$ defined via $d_{ii} = {\rm val}_{X}(v_{i})$.
\item The matrix $D - A$ is called the Laplacian matrix attached to $X$ and is denoted by $Q$.
\end{enumerate}
\end{definition}

The Ihara zeta function of a multigraph $X$ is defined to be
$$\zeta_{X}(u) = \prod_{\mathfrak{c}}(1-u^{l(\mathfrak{c})})^{-1}, $$
where the product is over all primes $\mathfrak{c}$ of $X$, and $l(\mathfrak{c})$ denotes the length of the prime $\mathfrak{c}$.  This product is usually infinite and converges if $|u|$ is small enough.  It is the reciprocal of a polynomial in $u$, and more precisely we have:
\begin{theorem}[Three-term determinant formula] \label{three_term}
Let $X$ be a multigraph, $A$ the adjacency matrix and $D$ the valency matrix of $X$.  Then, we have
$$\frac{1}{\zeta_{X}(u)} = (1-u^{2})^{-\chi(X)} \cdot {\rm det}(I - Au + (D - I)u^{2}). $$
\end{theorem}
\begin{proof}
See Theorem $2.5$ of \cite{Terras:2011}.  
\end{proof}
From now on, we let
$$Z_{X}(u) = \frac{1}{\zeta_{X}(u)} \in \mathbb{Z}[u]$$
and
$$h_{X}(u) =  {\rm det}(I - Au + (D - I)u^{2}) \in \mathbb{Z}[u].$$
Note that ${\rm deg}(Z_{X}(u)) = 2 |E_{X}|$, whereas ${\rm deg}(h_{X}(u))  =2 g_{X}$.  

In this paper, we will be solely interested in abelian covers $Y/X$ of multigraphs, that is Galois covers with abelian group of deck transformations.  We will use the notation ${\rm Gal}(Y/X)$ to denote the group of deck transformations and we shall refer to ${\rm Gal}(Y/X)$ as the Galois group of $Y/X$.  If $X^{ab}$ is a maximal abelian cover of $X$ ($X^{ab}$ is not a finite multigraph), then we have an isomorphism
\begin{equation} \label{rec_map}
\rho:H_{1}(X,\mathbb{Z}) \longrightarrow {\rm Gal}(X^{ab}/X). 
\end{equation}
So, abelian covers of a multigraph $X$ correspond to subgroups of finite index in $H_{1}(X,\mathbb{Z})$.  More on this can be found in \cite{Sunada:2013}.

Let $Y/X$ be an abelian cover of multigraphs, and set $G = {\rm Gal}(Y/X)$.  The cardinality of $G$ will be referred to as the degree of the cover.  If $\chi \in \widehat{G}$, then the Artin-Ihara $L$-function is defined by the formal infinite product
$$L_{X}(u,\chi) = \prod_{\mathfrak{c}} \left(1 - \chi\left(\left(\frac{Y/X}{\mathfrak{c}}\right) \right)u^{l(\mathfrak{c})} \right)^{-1}, $$
where the product is over all primes $\mathfrak{c}$ in $X$, $l(\mathfrak{c})$ denotes the length of $\mathfrak{c}$, and $\left(\frac{Y/X}{\mathfrak{c}}\right)$ the Frobenius automorphism at $\mathfrak{c}$.  For more details, see \cite{Terras:2011}.  As for the Ihara zeta function, this product is usually infinite and it can be shown to converge when $|u|$ is small enough.  Throughout this paper, we let $\chi_{0}$ be the trivial character of the group $G$.  Note that $L_{X}(u,\chi_{0}) = \zeta_{X}(u)$.
These $L$-functions turn out to be reciprocals of polynomials as Theorem \ref{three_term_L} below shows.  From now on, we fix a vertex $w_{i}$ of $Y$ in the fiber of $v_{i}$ for each $i=1,\ldots,g_{X}$.
\begin{definition} \label{artinian}
Let $Y/X$ be an abelian cover of multigraphs with automorphism group $G$.
\begin{enumerate}
\item For $\sigma \in G$, we let the matrix $A(\sigma)$ to be the $g_{X} \times g_{X}$ matrix $A(\sigma)=(a_{ij}(\sigma))$ defined via
\begin{equation*}
a_{ij}(\sigma) = 
\begin{cases}
\text{Twice the number of undirected loops at the vertex }w_{i}, &\text{ if } i=j \text{ and } \sigma = 1;\\
\text{The number of undirected edges connecting $w_{i}$ to $w_{j}^{\sigma}$}, &\text{ otherwise}.
\end{cases}
\end{equation*}
\item If $\chi \in \widehat{G}$, then we let
$$A_{\chi} = \sum_{\sigma \in G} \chi(\sigma) \cdot A(\sigma). $$
\end{enumerate}
\end{definition}
It is important to notice that the matrices $A_{\chi}$ are hermitian.  Therefore, they are diagonalizable with real eigenvalues.  We can now state the following important theorem.
\begin{theorem}[Three-term determinant formula for $L$-functions] \label{three_term_L}
Let $Y/X$ be an abelian cover of multigraphs with automorphism group $G$ and let $\chi \in \widehat{G}$.  Then, we have
$$\frac{1}{L_{X}(u,\chi)} = (1-u^{2})^{-\chi(X)} \cdot {\rm det}(I - A_{\chi}u + (D - I)u^{2}).$$
\end{theorem}
\begin{proof}
See Theorem $18.15$ of \cite{Terras:2011}.  
\end{proof}
From now on, we set
$$Z_{X}(u,\chi) = \frac{1}{L_{X}(u,\chi)} \in \mathbb{Z}[\chi][u] $$
and 
$$h_{X}(u,\chi) =  {\rm det}(I - A_{\chi}u + (D - I)u^{2}) \in \mathbb{Z}[\chi][u]. $$
Note that ${\rm deg}(Z_{X}(u,\chi)) = 2 |E_{X}|$ and ${\rm deg}(h_{X}(u,\chi)) = 2 g_{X}$ as well.

The Ihara zeta function of the cover factors into the product of the various $L$-functions as explained in the following theorem.
\begin{theorem}  \label{product_form}
Let $Y/X$ be an abelian cover of multigraphs.  Then one has
$$\zeta_{Y}(u) = \zeta_{X}(u) \cdot \prod_{\chi \neq \chi_{0}} L_{X}(u,\chi). $$
\end{theorem}
\begin{proof}
See Corollary $18.11$ in \cite{Terras:2011}.
\end{proof}
Note that one has $\chi(Y) = |G| \cdot \chi(X)$ and therefore, we also have
\begin{equation} \label{product_formula}
h_{Y}(u) = h_{X}(u) \prod_{\chi \neq \chi_{0}} h_{X}(u,\chi).
\end{equation}

The Artin-Ihara $L$-functions satisfy the usual Artin formalism.  (See Proposition $18.10$ in \cite{Terras:2011}.)  So, in particular we have the following useful result.
\begin{proposition} \label{inflation}
Let $Y/X$ be an abelian cover of multigraphs with Galois group $G$ and let $H$ be a subgroup of $G$.  Let also $\Gamma = G/H$.  If $\chi \in \widehat{G}$ is such that $H \subseteq {\rm ker}(\chi)$, then $\chi$ induces a character $\widetilde{\chi} \in \widehat{\Gamma}$.  One has
$$h_{X}(u,\chi) = h_{X}(u,\widetilde{\chi}). $$
\end{proposition}
\begin{proof}
This follows from the corresponding property of Artin-Ihara $L$-functions and noticing that 
$$\frac{Z_{X}(u,\chi)}{h_{X}(u,\chi)} = (1-u^{2})^{-\chi(X)} = \frac{Z_{X}(u,\widetilde{\chi})}{h_{X}(u,\widetilde{\chi})}. $$
\end{proof}

\section{The number of spanning trees in abelian covers of multigraphs} \label{number}
As before, $X$ denotes a multigraph.  Since the Laplacian $Q = D- A$ is a singular matrix, we have $h_{X}(1) = 0$.
\begin{theorem} \label{special_value}
Let $X$ be a multigraph.  One has
$$h_{X}'(1) = -2 \chi(X) \kappa_{X}, $$
and if $\chi(X) = 0$, then $h_{X}''(1) = 2 g_{X}^{2}$.
\end{theorem}
\begin{proof}
See \cite{Northshield:1998}.
\end{proof}
Since we are assuming that our multigraphs do not have vertices of degree one, we have $\chi(X) = 0$ if and only if $X = C_{g}$, the cycle graph on $g$ vertices.  If $\chi(X) = 0$, then $b_{1}(X) = 1$ and thus $H_{1}(X,\mathbb{Z}) \simeq \mathbb{Z}$.  It follows from (\ref{rec_map}) that for every $m \ge 1$, there is a unique abelian cover of $X$ of degree $m$ and its Galois group is isomorphic to $\mathbb{Z}/m\mathbb{Z}$.  If $Y/X$ is such a cover and $X = C_{g}$,  then $Y$ is isomorphic to the cycle graph $C_{|G|g}$, since $\chi(Y) = |G| \cdot \chi(X)$.

If $Y/X$ is an abelian cover, then (\ref{product_formula}) implies 
$$h_{Y}'(1) = h_{X}'(1) \prod_{\chi \neq \chi_{0}} h_{X}(1,\chi). $$
In particular, if $\chi$ is a non-trivial character, then $h_{X}(1,\chi) \neq 0$ when $\chi(X)\neq 0$.  (This is also true if $\chi(X) = 0$, but not because of this last equation.)  If $\chi(X) \neq 0$, using Theorem \ref{special_value}, one has
$$|G| \cdot \kappa_{Y} = \kappa_{X} \prod_{\chi \neq \chi_{0}}h_{X}(1,\chi), $$
and since $h_{X}(u,\chi) \in \mathbb{Z}[\chi][u]$, we have $h_{X}(1,\chi) \in \mathbb{Z}[\chi]$. 

From now on, we consider the  algebraic closure $\overline{\mathbb{Q}}$ of $\mathbb{Q}$ in $\mathbb{C}$.  The absolute Galois group ${\rm Gal}(\overline{\mathbb{Q}}/\mathbb{Q})$ will be denoted by $G_{\mathbb{Q}}$.  The group $G_{\mathbb{Q}}$ acts on $\widehat{G}$ via $\chi^{\tau} = \tau \circ \chi$ whenever $\tau \in G_{\mathbb{Q}}$.  Two characters $\chi_{1}, \chi_{2} \in \widehat{G}$ are in the same orbit if and only if
$${\rm ker}(\chi_{1}) = {\rm ker}(\chi_{2}). $$
An orbit will usually be denoted by a symbol such as $\Psi$.  The orbit of the trivial character $\chi_{0}$ consists of the trivial character only and will be denoted by $\Psi_{0}$.

Given an orbit $\Psi$, we let 
$$h_{X}(u,\Psi) = \prod_{\chi \in \Psi}h_{X}(u,\chi). $$
Then $h_{X}(u,\Psi) \in \mathbb{Z}[u]$ and $h_{X}(1,\Psi) \in \mathbb{Z}$.  Furthermore, we have the following decomposition of $|G| \cdot \kappa_{Y}$ as a product of integers:
\begin{equation} \label{integer_decomposition}
|G| \cdot \kappa_{Y} = \kappa_{X} \prod_{\Psi \neq \Psi_{0}} h_{X}(1,\Psi).
\end{equation}
In this last equation, the product is over all orbits $\Psi \in G_{\mathbb{Q}} \backslash \widehat{G}$ satisfying $\Psi \neq \Psi_{0}$, and we are still assuming $\chi(X) \neq 0$.

\section{An upper bound for the number of spanning trees in regular abelian $\ell$-towers} \label{upper_bound_1}
Let us start with the following definition, where from now on, $\ell$ denotes a rational prime number.
\begin{definition}
An abelian $\ell$-tower of multigraphs above a multigraph $X$ is a sequence of covers
$$X = X_{0} \longleftarrow X_{1} \longleftarrow X_{2} \longleftarrow \ldots \longleftarrow X_{n} \longleftarrow \ldots$$
such that:
\begin{enumerate}
\item For each $n \ge 0$, the cover $X_{n+1}/X_{n}$ is abelian with Galois group isomorphic to $\mathbb{Z}/\ell\mathbb{Z}$,
\item For each $n \ge 0$, the cover $X_{n}/X$ is abelian with Galois group isomorphic to $\mathbb{Z}/\ell^{n}\mathbb{Z}$.
\end{enumerate}
\end{definition}
At the layer $n$ of an abelian $\ell$-tower, we will write $\kappa_{n}$ instead of $\kappa_{X_{n}}$ and similarly for any other notation involving $X_{n}$ as an index.  Note that since $H_{1}(X,\mathbb{Z}) \simeq \mathbb{Z}^{b_{1}(X)}$, such abelian $\ell$-towers always exist for every multigraph $X$.  If $X = C_{g}$, then there exists a unique such abelian $\ell$-tower
$$X = X_{0} \longleftarrow X_{1} \longleftarrow X_{2} \longleftarrow \ldots \longleftarrow X_{n} \longleftarrow \ldots,$$
where $X_{n} = C_{\ell^{n}g}$.  In this case, we have $\kappa_{n} = \ell^{n}g$ and 
$${\rm ord}_{\ell}(\kappa_{n}) =  n + {\rm ord}_{\ell}(g). $$

Let now $q$ be a positive integer.  Recall that a multigraph $X$ is called $(q+1)$-regular if
$${\rm val}_{X}(v) = q+1$$
for all $v \in V_{X}$.  A multigraph $X$ is $2$-regular if and only if $X$ a cycle graph if and only if $\chi(X) = 0$. Given a multigraph $X$, we let ${\rm Spec}(X)$ be the set of eigenvalues of the adjacency matrix of $X$.  These eigenvalues are necessarily real, since the adjacency matrix is symmetric.  If $X$ is $(q+1)$-regular, then it is known that $(q+1) \in {\rm Spec}(X)$ with multiplicity one, and also that if $\lambda \in {\rm Spec}(X)$, then
$$|\lambda| \le q+1.$$
For a multigraph $X$ on $g_{X}$ vertices, we will always label the eigenvalues as follows:
$$-(q+1) \le \lambda_{1} \le \ldots \le \lambda_{g_{X}-1} < \lambda_{g_{X}} = q + 1. $$
For a $(q+1)$-regular multigraph $X$, we have
\begin{equation} \label{h_regular}
\begin{aligned}
h_{X}(u) &= {\rm det}(I - Au + qIu^{2}) \\
&= \prod_{i=1}^{g_{X}} (1- \lambda u + q u^{2}).
\end{aligned}
\end{equation}
If $Y/X$ is an abelian cover of multigraphs, and $X$ is $(q+1)$-regular, then so is $Y$, since the projection map is a local homeomorphism.  Furthermore, if $\chi \in \widehat{G}$, then
\begin{equation*}
\begin{aligned}
h_{X}(u,\chi) &= {\rm det}(I - A_{\chi}u + qIu^{2}) \\
&= \prod_{i=1}^{g_{X}}(1 - \lambda_{\chi,i}u + qu^{2}),
\end{aligned}
\end{equation*}
where the $\lambda_{\chi,i}$ are the real eigenvalues of $A_{\chi}$ (recall that the matrices $A_{\chi}$ are hermitian).  In particular, we have
$$h_{X}(u,\chi) \in \mathbb{Z}[\chi]^{+}[u] \text{ and } h_{X}(1,\chi) \in \mathbb{Z}[\chi]^{+}. $$
If 
$$X = X_{0} \longleftarrow X_{1} \longleftarrow X_{2} \longleftarrow \ldots \longleftarrow X_{n} \longleftarrow \ldots$$
is an abelian $\ell$-tower of a $(q+1)$-regular multigraph $X$, then
$${\rm Spec}(X_{n}) \subseteq [-(q+1),q+1], $$
for all $n \ge 0$.  Such an abelian $\ell$-tower of multigraphs will be called a regular abelian $\ell$-tower.

\begin{theorem} \label{upper_bound}
Let $X$ be a $(q+1)$-regular multigraph such that $\chi(X) \neq 0$, and assume that
$$X = X_{0} \longleftarrow X_{1} \longleftarrow X_{2} \longleftarrow \ldots \longleftarrow X_{n} \longleftarrow \ldots$$
is a regular abelian $\ell$-tower over $X$.  Then
$$\kappa_{n} \le A \cdot \frac{(2(q+1))^{\ell^{n}g}}{\ell^{n}},$$
where $g = g_{X}$ and
$$A = \frac{1}{4 |\chi(X)|} \frac{q-1}{q+1}. $$
Thus, in particular
$${\rm ord}_{\ell}(\kappa_{n}) \le g \log_{\ell}(2(q+1)) \ell^{n} - n + \log_{\ell}(A).$$
\end{theorem}
\begin{proof}
Using (\ref{h_regular}), we have
$$h_{n}(u) = \prod_{i=1}^{g_{n}}(1-\lambda_{i}u + qu^{2}). $$
Differentiating this last equation and evaluating at $1$ gives
$$h_{n}'(1) = (q-1) \prod_{i=1}^{g_{n} - 1}(1 - \lambda_{i} + q), $$
and thus
$$2|\chi(X_{n})| \kappa_{n} = (q-1) \prod_{i=1}^{g_{n} - 1}(1 - \lambda_{i} + q).$$
Therefore,
$$2 \ell^{n} |\chi(X)| \kappa_{n} \le \frac{1}{2} \frac{q-1}{q+1} (2(q+1))^{\ell^{n}g}.$$
The result follows as once.
\end{proof}

\section{Regular abelian $\ell$-towers of bouquets} \label{over_bouquet}
Throughout this section, we assume that $g_{X} = 1$ so that $X$ is a bouquet.
\subsection{A lower bound} \label{lower_bound_1}
Let $q$ be an odd integer greater than $2$ and assume that $X$ is $(q+1)$-regular.  Thus $X$ is a bouquet with $(q+1)/2$ loops and $\chi(X) = (1-q)/2 \neq 0$.  We take an abelian $\ell$-tower
$$X = X_{0} \longleftarrow X_{1} \longleftarrow X_{2} \longleftarrow \ldots \longleftarrow X_{n} \longleftarrow \ldots$$
above the bouquet $X$.  

We let $G_{n} = {\rm Gal}(X_{n}/X) \simeq \mathbb{Z}/\ell^{n} \mathbb{Z}$.  Two characters $\chi_{1}, \chi_{2} \in \widehat{G}_{n}$ satisfy ${\rm ker}(\chi_{1}) = {\rm ker}(\chi_{2})$ if and only if they have the same order.  The orbits of $\widehat{G}_{n}$ under the actions of $G_{\mathbb{Q}}$ will be denoted by $\Psi_{i,n}$ for $i=0,\ldots,n$.  So $\Psi_{i,n}$ consists of the characters of $G_{n}$ of order precisely $\ell^{i}$.  We have
$$|\Psi_{i,n}| = \varphi(\ell^{i}). $$
Formula (\ref{integer_decomposition}) becomes then
\begin{equation} \label{integer_decomposition_prime}
\ell^{n} \cdot \kappa_{n} = \prod_{i=1}^{n} h_{X}(1,\Psi_{i,n}),
\end{equation}
since $\kappa_{X} = 1$.  By Proposition \ref{inflation}, for any $i$ and $n$ satisfying $i=0,1,\ldots,n$, we have
$$h_{X}(u,\Psi_{i,i}) = h_{X}(u,\Psi_{i,n}).$$
Therefore, (\ref{integer_decomposition_prime}) gives
\begin{equation}
\begin{aligned}
{\rm ord}_{\ell}(\kappa_{n}) + n &= \sum_{i=1}^{n} {\rm ord}_{\ell}(h_{X}(1,\Psi_{i,n})) \\
&= \sum_{i=1}^{n} {\rm ord}_{\ell}(h_{X}(1,\Psi_{i,i}))
\end{aligned}
\end{equation}
In other words, we are reduced to studying the orbits consisting of the faithful characters at each layer.  

If $\chi \in \Psi_{i,i}$, then $\mathbb{Q}(\chi) = \mathbb{Q}(\zeta_{\ell^{i}})$, and
\begin{equation} \label{simple_for}
h_{X}(1,\chi) = 1 - \lambda_{\chi} + q \in \mathbb{Z}[\zeta_{\ell^{i}}]^{+},
\end{equation}
where $\lambda_{\chi} = A_{\chi}$ of Definition \ref{artinian}.

Since $\mathcal{L}_{i}$ is the unique prime lying above $\ell$, we have
$${\rm ord}_{\mathcal{L}_{i}}(h_{X}(1,\chi)) = {\rm ord}_{\mathcal{L}_{i}}(h_{X}(1,\chi^{\tau})) $$
for all $\tau \in G_{\mathbb{Q}}$.  Putting this together, we obtain
\begin{equation*}
\begin{aligned}
{\rm ord}_{\ell}(h_{X}(1,\Psi_{i,i}))  &= \frac{1}{\varphi(\ell^{i})}{\rm ord}_{\mathcal{L}_{i}}\left(\prod_{\chi \in \Psi_{i,i}}h_{X}(1,\chi) \right)\\
&= {\rm ord}_{\mathcal{L}_{i}}(h_{X}(1,\chi)),
\end{aligned}
\end{equation*}
where $\chi$ is any character in $\Psi_{i,i}$.  It follows that if we choose a faithful character $\chi_{i}$ of $\widehat{G}_{i}$ for each $i \ge 1$, we have
\begin{equation} \label{starting_point}
{\rm ord}_{\ell}(\kappa_{n}) = -n + \sum_{i=1}^{n}{\rm ord}_{\mathcal{L}_{i}}(h_{X}(1,\chi_{i})).
\end{equation}
If we assume furthermore that $\ell \neq 2$, then we get
\begin{equation} \label{extra_two}
{\rm ord}_{\ell}(\kappa_{n}) =  -n + 2\sum_{i=1}^{n}{\rm ord}_{\mathcal{L}_{i}^{+}}(h_{X}(1,\chi_{i})),
\end{equation}
whereas, if $\ell = 2$, we get instead
\begin{equation} \label{extra_two_2}
{\rm ord}_{\ell}(\kappa_{n}) =  -n + {\rm ord}_{\mathcal{L}_{1}}(h_{X}(1,\chi_{1})) + 2\sum_{i=2}^{n}{\rm ord}_{\mathcal{L}_{i}^{+}}(h_{X}(1,\chi_{i})).
\end{equation}

\begin{corollary}
In any abelian $\ell$-tower
$$X = X_{0} \longleftarrow X_{1} \longleftarrow X_{2} \longleftarrow \ldots \longleftarrow X_{n} \longleftarrow \ldots$$
of a bouquet, one has
$${\rm ord}_{\ell}(\kappa_{n}) \ge n.$$
\end{corollary}
\begin{proof}
We assume first that $\ell$ is odd.  We claim that
$${\rm ord}_{\mathcal{L}_{n}}(h_{X}(1,\chi_{n})) \ge 1 $$
for all $n \ge 1$.  Indeed, first we have 
\begin{equation} \label{case_un}
{\rm ord}_{\mathcal{L}_{1}}(h_{X}(1,\chi_{1})) - 1 \ge 0, 
\end{equation}
by (\ref{starting_point}) with $n = 1$.  Now, for any abelian cover $Y/X$ of multigraphs, we have $\kappa_{X} \, | \, \kappa_{Y}$ by Corollary $4.15$ of \cite{Baker/Norine:2009}.  (See also Corollary $4.10$ of \cite{Hammer:2020}.)  Thus, ${\rm ord}_{\ell}(\kappa_{n}) \le {\rm ord}_{\ell}(\kappa_{n+1})$,  and we have
$$-n + \sum_{i=1}^{n}{\rm ord}_{\mathcal{L}_{i}}(h_{X}(1,\chi_{i})) \le -(n+1) + \sum_{i=1}^{n+1}{\rm ord}_{\mathcal{L}_{i}}(h_{X}(1,\chi_{i})). $$
Therefore,
$$ {\rm ord}_{\mathcal{L}_{n+1}}(h_{X}(1,\chi_{n+1})) - 1 \ge 0, $$
for all $n \ge 1$.  Combined with (\ref{case_un}), we get
$$ {\rm ord}_{\mathcal{L}_{n}}(h_{X}(1,\chi_{n})) \ge 1, $$
for all $n \ge 1$.

Now, since
$${\rm ord}_{\mathcal{L}_{i}}(x) = 2 \cdot {\rm ord}_{\mathcal{L}_{i}^{+}}(x)$$
for all $x \in \mathbb{Q}(\zeta_{\ell^{i}})^{+}$, we also have
$${\rm ord}_{\mathcal{L}_{i}^{+}}(h_{X}(1,\chi_{i})) \ge 1. $$
Using (\ref{extra_two}) gives the desired result.

If $\ell = 2$, then we claim first that $2 \, | \, \kappa_{1}$.  Indeed, $X_{1}$ has two vertices and if $\kappa_{1}$ were odd, then the valency of any of the two vertices would be odd, since $\kappa_{1}$ is the number of simple undirected edges between the two vertices.  But this is a contradiction.  Thus, from (\ref{starting_point}) with $n = 1$, we get
$${\rm ord}_{\mathcal{L}_{1}}(h_{X}(1,\chi_{1})) \ge 2. $$
By the same token as above, we deduce our result using (\ref{extra_two_2}) instead of (\ref{extra_two}).
\end{proof}

\subsection{Cayley-Serre multigraphs} \label{cayley_serre}
The Galois covers of bouquets are precisely the Cayley-Serre multigraphs, so we remind the reader about them here.  Let $G$ be a finite group, $S$ a finite set, and
$$i:S\longrightarrow G $$
a function.  Then, the Cayley-Serre multigraph $X(G,S,i)$ associated to the data $(G,S,i)$ is defined as follows.  The set of vertices is $\{v_{\sigma} \, | \, \sigma \in G \}$.  The set of edges is $\{ (\sigma,s) \, | \, \sigma \in G \text{ and } s \in S\}$.  The edge $(\sigma,s)$ connects the vertex $v_{\sigma}$ to $v_{\sigma \cdot i(s)}$.  The resulting multigraph is a directed multigraph, but we forget about the directions and we obtain a multigraph, denoted by $X(G,S,i)$, which is connected if and only if $i(S)$ generates $G$.

A convenient way to construct regular abelian $\ell$-covers of bouquets is as follows.  Let $\mathbb{Z}_{\ell}$ denote the ring of $\ell$-adic integers.  Let $S$ be a finite set and let
$$i:S \longrightarrow \mathbb{Z}_{\ell} $$
be a function.  Then for each $n \ge 1$, we consider the function 
$$i_{n}:S \longrightarrow \mathbb{Z}/\ell^{n} \mathbb{Z}$$ 
obtained from the composition
$$S  \stackrel{i}{\longrightarrow} \mathbb{Z}_{\ell} \longrightarrow \mathbb{Z}_{\ell}/\ell^{n} \mathbb{Z}_{\ell} \stackrel{\simeq}{\longrightarrow} \mathbb{Z}/\ell^{n}\mathbb{Z}.  $$
If we assume that there is at least one element $s \in S$ such that $s \in \mathbb{Z}_{\ell}^{\times}$, then all Cayley-Serre multigraphs are connected, and we get a regular abelian $\ell$-cover of $B_{|S|}$:
$$X = B_{|S|} \longleftarrow X(\mathbb{Z}/\ell \mathbb{Z},S,i_{1}) \longleftarrow X(\mathbb{Z}/\ell^{2}\mathbb{Z},S,i_{2}) \longleftarrow \ldots \longleftarrow X(\mathbb{Z}/\ell^{n}\mathbb{Z},S,i_{n}) \longleftarrow \ldots$$

\subsection{The spectra of Cayley-Serre multigraphs} \label{spectra_cs}
The Cayley-Serre multigraphs are particularly pleasant, since their spectra can be found explicitly.  From now on, we assume that $G$ is a finite abelian group with an additive binary operation.  We let $L(G)$ be the set of complex-valued functions on $G$.  It is well-known that $L(G)$ is finite dimensional as a $\mathbb{C}$-vector space and that $\widehat{G}$ forms a basis for $L(G)$.  Another basis of $L(G)$ is given by the functions $\chi_{\sigma}$, where 
\begin{equation*}
\chi_{\sigma}(\tau) = 
\begin{cases}
1, &\text{ if } \sigma = \tau;\\
0, &\text{ otherwise}.
\end{cases}
\end{equation*}
Given the data $(G,S,i)$ defining a Cayley-Serre multigraph, we define $\phi:L(G) \longrightarrow L(G)$ via $f \mapsto \phi(f)$, where
$$\phi(f)(\sigma) = \sum_{s \in S}\Big(f(\sigma + i(s)) +  f(\sigma - i(s))\Big). $$
It is simple to check that $\phi$ is a $\mathbb{C}$-linear map.  Furthermore, the matrix of $\phi$ with respect to the basis $\{\chi_{\sigma} \, | \, \sigma \in G\}$ is the adjacency matrix of $X(G,S,i)$.  In other words, one can see the adjacency matrix of $X(G,S,i)$ as an operator on $L(G)$.  Now, for $\chi \in \widehat{G}$, we have
\begin{equation*}
\begin{aligned}
\phi(\chi)(\sigma) &= \sum_{s \in S}\Big(\chi(\sigma + i(s)) +  \chi(\sigma - i(s))\Big) \\
&= \sum_{s \in S}\Big(\chi(\sigma)\chi(i(s)) + \chi(\sigma)\overline{\chi}(i(s)) \Big) \\
&= \left(\sum_{s \in S}\Big(\chi(i(s)) + \overline{\chi}(i(s))\Big) \right) \cdot \chi(\sigma).
\end{aligned}
\end{equation*}
Thus, $\chi$ is an eigenvector in $L(G)$ with eigenvalue
\begin{equation} \label{eigenvalues}
\sum_{s \in S}\Big(\chi(i(s)) + \overline{\chi}(i(s))\Big).
\end{equation}
It follows that the eigenvalues in ${\rm Spec}(X(G,S,i))$ are precisely the expressions (\ref{eigenvalues}), where $\chi$ runs over all characters of $G$.  A Cayley-Serre multigraph $X(G,S,i)$ for a finite abelian group $G$ is an abelian cover of $B_{|S|}$ and we can now show the following result.
\begin{theorem} \label{match_eigenvalues}
Let $G$ be a finite abelian group and let $(G,S,i)$ be some data defining a connected Cayley-Serre multigraph $X(G,S,i)$.  If $\chi \in \widehat{G}$, then
$$A_{\chi} = \sum_{s \in S}\Big(\chi(i(s)) + \overline{\chi}(i(s))\Big), $$
where $A_{\chi}$ was defined in Definition \ref{artinian}.
\end{theorem}
\begin{proof}
First, we note that the matrix $A(\sigma) = (a(\sigma))$ of Definition \ref{artinian} is a one by one matrix and it satisfies
$$a(\sigma) = b(\sigma) + c(\sigma), $$
where 
$$b(\sigma) = |\{s \in S \, | \, \sigma = i(s) \}| \text{ and } c(\sigma) = |\{s \in S \, | \, \sigma = -i(s)\}|.$$
Now,
$$\sum_{s \in S}\chi(i(s)) = \sum_{\sigma \in G} \sum_{\substack{s \in S \\ \sigma = i(s)}}\chi(\sigma) = \sum_{\sigma \in G} \chi(\sigma) b(\sigma), $$
and similarly
$$\sum_{s \in S}\overline{\chi}(i(s)) = \sum_{\sigma \in G} \sum_{\substack{s \in S \\ \sigma = -i(s)}}\chi(\sigma) = \sum_{\sigma \in G} \chi(\sigma) c(\sigma). $$
Therefore,
$$A_{\chi} = \sum_{\sigma \in G}\chi(\sigma)a(\sigma) =   \sum_{s \in S}\Big(\chi(i(s)) + \overline{\chi}(i(s))\Big),$$
and this is what we wanted to show.
\end{proof}

If we denote by $\lambda_{\chi}$ the eigenvalue (\ref{eigenvalues}) corresponding to $\chi \in \widehat{G}$, and $|S| = (q+1)/2$, then we have
\begin{equation} \label{link_alg}
\begin{aligned}
1 - \lambda_{\chi} + q &= 2|S| - \lambda_{\chi} \\
&= 2|S| - \sum_{s \in S} (\chi(i(s)) + \overline{\chi}(i(s))) \\
&= \sum_{s \in S} \left(1 - \chi(i(s)) + (1 - \overline{\chi}(i(s))) \right) \\
&= \sum_{s \in S}(1- \chi(i(s)))(1 - \overline{\chi}(i(s))),
\end{aligned}
\end{equation}
so we are led to study the algebraic integers of the form
$$(1-\zeta_{m}^{a})(1 - \zeta_{m}^{-a}), $$
where $m$ is an integer greater than or equal to one, and $a$ is any integer.

\subsection{A few lemmas about some cyclotomic algebraic integers} \label{cyclo_integer}
Let $m \ge 1$ be an integer, and let $a$ be any integer.  We define
$$\varepsilon_{m}(a) = (1 - \zeta_{m}^{a})(1-\zeta_{m}^{-a}), $$
where
$$\zeta_{m} = \exp\left(\frac{2 \pi i}{m} \right). $$
We write simply $\varepsilon_{m}$ rather than $\varepsilon_{m}(1)$.
\begin{lemma} \label{useful_form}
For any integer $a \ge 1$ and any $m \ge 2$, we have
$$\frac{\varepsilon_{m}(a)}{\varepsilon_{m}} = a^{2} - (a-1) \varepsilon_{m} - (a-2) \varepsilon_{m}(2) - \ldots - \varepsilon_{m}(a-1). $$
\end{lemma}
\begin{proof}
We calculate
\begin{equation*}
\begin{aligned}
\frac{\varepsilon_{m}(a)}{\varepsilon_{m}} &= \frac{(1- \zeta_{m}^{a})(1- \bar{\zeta}_{m}^{a})}{(1-\zeta_{m})(1 - \bar{\zeta}_{m})} \\
&= (1 + \zeta_{m} + \zeta_{m}^{2} + \ldots + \zeta_{m}^{a-1})(1+\bar{\zeta}_{m} + \bar{\zeta}_{m}^{2} + \ldots + \bar{\zeta}_{m}^{a-1}) \\
&= \sum_{i=0}^{a-1} \sum_{j=0}^{a-1} \zeta_{m}^{i - j} \\
&= a + \sum_{0 \le i < j \le a-1}(\zeta_{m}^{i-j} + \zeta_{m}^{j-i}) \\
&= a^{2} + \sum_{0 \le i < j \le a-1} (\zeta_{m}^{j-i} - 1)(1 - \zeta_{m}^{i-j}) \\
&= a^{2} - (a-1) \varepsilon_{m} - (a-2) \varepsilon_{m}(2) - \ldots - \varepsilon_{m}(a-1).
\end{aligned}
\end{equation*}
\end{proof}
We clearly have
$$\varepsilon_{1}(a) = 0 \text{ and } \varepsilon_{m}(0) = 0$$
for all $a \in \mathbb{Z}$ and all $m \ge 1$.  Also
$$\varepsilon_{m}(a) = \varepsilon_{m}(-a),$$
for all $m \ge 1$ and all $a \in \mathbb{Z}$.  For us, $m$ will always be a power of a prime, so we let $m = \ell^{n}$, where $\ell$ is a rational prime and $n$ is a non-negative integer.  Note that $\varepsilon_{\ell^{n}}(a) \in \mathbb{Z}[\zeta_{\ell^{n}}]^{+}$ for all $a \in \mathbb{Z}$.  Furthermore, from (\ref{unique_prime}), we have
\begin{equation} \label{valuation}
\begin{aligned}
{\rm ord}_{\mathcal{L}_{n}}(\varepsilon_{\ell^{n}}) &= {\rm ord}_{\mathcal{L}_{n}}(\mathcal{L}_{n}^{2})\\
&=2,
\end{aligned}
\end{equation}
provided $n \ge 1$.  We record in the next lemma a few simple properties satisfied by these algebraic integers.
\begin{lemma} \label{valua}
With the notation as above:
\begin{enumerate}
\item For all $a \ge 1$, ${\rm ord}_{\mathcal{L}_{n}}(\varepsilon_{\ell^{n}}(a)) \ge 2$, \label{un}
\item If $(a,\ell) = 1$, then ${\rm ord}_{\mathcal{L}_{n}}(\varepsilon_{\ell^{n}}(a)) = 2$, \label{deux}
\item If $a = \ell^{s}b$ with $(b,\ell)=1$ and $s < n$, then ${\rm ord}_{\mathcal{L}_{n}}(\varepsilon_{\ell^{n}}(a)) = 2\ell^{s}$, \label{trois}
\item If $\ell^{n} \, | \, a$, then ${\rm ord}_{\mathcal{L}_{n}}(\varepsilon_{\ell^{n}}(a)) = \infty$. \label{quatre}
\end{enumerate}
\end{lemma}
\begin{proof}
Point ($\ref{un}$) follows from Lemma $\ref{useful_form}$ and ($\ref{valuation}$).  Point ($\ref{deux}$) follows from the first claim and from the equality ($\ref{equality}$).  Point ($\ref{trois}$) follows from the transitivity property of the ramification indices and from point ($\ref{deux}$).  Point ($\ref{quatre}$) is clear, since $\varepsilon_{\ell^{n}}(a) = 0$ if $\ell^{n} \, | \, a$.
\end{proof}
We now define some polynomials with integer coefficients recursively as follows.  For $a=0,1$, we let
$$P_{0}(T) = 0 \text{ and } P_{1}(T) = T.$$
If $a \ge 2$, then we set 
$$P_{a}(T) = T(a^{2} - (a-1)P_{1}(T) - (a-2)P_{2}(T) - \ldots - P_{a-1}(T)).$$
We record here a few simple properties of these polynomials.
\begin{lemma} \label{pol_prop}
With the notation as above, for all $a \ge 0$:
\begin{enumerate}
\item We have $P_{a}(0) = 0$,
\item If $a \ge 1$, the degree of $P_{a}(T)$ is $a$,
\item The coefficient of $T$ is $a^{2}$,
\item The leading coefficients of $P_{a}(T)$ is $(-1)^{a+1}$,
\item For all $m \ge 1$, $P_{a}(\varepsilon_{m}) = \varepsilon_{m}(a) = \varepsilon_{m}(-a)$.
\end{enumerate}
\end{lemma}
\begin{proof}
All these properties follow by induction, and the details are left to the reader.
\end{proof}
The first few polynomials are
\begin{equation*}
\begin{aligned}
P_{1}(T) &= T \\
P_{2}(T) &= -T^{2} + 4T \\
P_{3}(T) &= T^{3} -6T^{2} + 9T \\
P_{4}(T) &= - T^{4} + 8T^{3} - 20T^{2} + 16T\\
P_{5}(T) &= T^{5} - 10T^{4} + 35T^{3} - 50T^{2} + 25T.
\end{aligned}
\end{equation*}

\subsection{A particular family of regular abelian $\ell$-towers of bouquets} \label{main_1}
We can now show the following theorem.
\begin{theorem} \label{main}
Let $q$ be an odd positive integer satisfying $q \ge 2$.  Let also $S =\{s_{1},\ldots,s_{t} \}$ be a finite set of integers with cardinality $t = (q+1)/2$, and let $i:S \longrightarrow \mathbb{Z}$ be a function.  For $j=1,\ldots, t$, let $a_{j}$ be the integer satisfying $i(s_{j})=a_{j}$.  Assume that at least one of $a_{1},\ldots,a_{t}$ is relatively prime with $\ell$ (so that our Cayley-Serre multigraphs are connected) and set $b_{j} = |a_{j}|$.  Consider the regular abelian $\ell$-cover
$$X = B_{t} \longleftarrow X(\mathbb{Z}/\ell \mathbb{Z},S,i_{1}) \longleftarrow X(\mathbb{Z}/\ell^{2}\mathbb{Z},S,i_{2}) \longleftarrow \ldots \longleftarrow X(\mathbb{Z}/\ell^{n}\mathbb{Z},S,i_{n}) \longleftarrow \ldots$$
and define the integers $c_{j}$ via 
\begin{equation*}
\begin{aligned}
Q(T) &= P_{b_{1}}(T) + \ldots + P_{b_{t}}(T) \\
&= c_{1}T + \ldots + c_{a}T^{a}.
\end{aligned}
\end{equation*}
Let
$$\mu_{\ell} = {\rm min}\{{\rm ord}_{\ell}(c_{j})\, | \, j =1,\ldots,a\}, $$
and 
$$\lambda_{\ell} = {\rm min}\{2j \, | \, j = 1,\ldots, a \text{ and } {\rm ord}_{\ell}(c_{j}) = \mu\}  -1. $$
Then, there exists $n_{0} \ge 1$ and a constant $\nu_{\ell}$ (depending also on the $a_{j}$) such that
$${\rm ord}_{\ell}(\kappa_{n}) = \mu_{\ell} \cdot \ell^{n} + \lambda_{\ell} \cdot n + \nu_{\ell},$$
when $n \ge n_{0}$. 
\end{theorem}
\begin{proof}
Let $\chi_{i} \in \widehat{G}_{i}$ be the faithful character defined via $\chi_{i}(\bar{1}) = \zeta_{\ell^{i}}$.  Then, putting (\ref{simple_for}), (\ref{starting_point}), Theorem \ref{match_eigenvalues}, and  (\ref{link_alg}) together, we obtain
$${\rm ord}_{\ell}(\kappa_{n}) = -n + \sum_{i=1}^{n} {\rm ord}_{\mathcal{L}_{i}}(\varepsilon_{\ell^{i}}(a_{1}) + \ldots +  \varepsilon_{\ell^{i}}(a_{t})). $$
Now, by Lemma \ref{pol_prop}, we have
\begin{equation*}
\begin{aligned}
\varepsilon_{\ell^{i}}(a_{1}) + \ldots +  \varepsilon_{\ell^{i}}(a_{t}) &= \varepsilon_{\ell^{i}}(b_{1}) + \ldots +  \varepsilon_{\ell^{i}}(b_{t}) \\
&= P_{b_{1}}(\varepsilon_{\ell^{i}}) + \ldots + P_{b_{t}}(\varepsilon_{\ell^{i}}) \\
&= Q(\varepsilon_{\ell^{i}}) \\
&= c_{1}\varepsilon_{\ell^{i}} + \ldots + c_{a}(\varepsilon_{\ell^{i}})^{a}.
\end{aligned}
\end{equation*}
If $i$ is large enough, then we get
$${\rm ord}_{\mathcal{L}_{i}}(c_{1}\varepsilon_{\ell^{i}} + \ldots + c_{a}(\varepsilon_{\ell^{i}})^{a}) = \mu_{\ell} \cdot \varphi(\ell^{i}) + (\lambda_{\ell} + 1).$$
Therefore, there exists $n_{0} \ge 1$ such that if $n \ge n_{0}$, then 
$$ {\rm ord}_{\ell}(\kappa_{n}) = -n + C + \sum_{i=n_{0}}^{n} (\mu_{\ell} \cdot \varphi(\ell^{i}) + (\lambda_{\ell} + 1)),$$
for some constant $C$.  But
$$\sum_{i=n_{0}}^{n} (\mu_{\ell} \cdot \varphi(\ell^{i}) + (\lambda_{\ell} + 1)) = (\lambda_{\ell} + 1)(n-n_{0}) + \mu_{\ell}(\ell^{n} - \ell^{n_{0}-1}).$$
The result follows at once.
\end{proof}

\begin{corollary}
With the same notation as in Theorem \ref{main}, if
$$\ell \nmid (a_{1}^{2} + \ldots + a_{t}^{2}), $$
then $\mu_{\ell} = 0$, $\lambda_{\ell} = 1$ and $\nu_{\ell}=0$, so that
$${\rm ord}_{\ell}(\kappa_{n}) = n $$
for all $n \ge 1$.
\end{corollary}
\begin{proof}
Indeed, one has
$$Q(T) = T (c_{1} + c_{2}T + \ldots + c_{a}T^{a-1}), $$
and thus by (\ref{equality}) and Lemma \ref{valua}, we get ${\rm ord}_{\mathcal{L}_{i}}(\varepsilon_{\ell^{i}}) = 2$ for all $i \ge 1$, since $c_{1} = a_{1}^{2}+ \ldots + a_{t}^{2}$.
\end{proof}
It follows that the lower bound we found in \S \ref{lower_bound_1} is the best possible one in general.

\begin{corollary}
With the same notation as in Theorem \ref{main}, we have
$$\mu_{\ell} \le \log_{\ell}(2(q+1)). $$
\end{corollary}
\begin{proof}
This follows from Theorem \ref{upper_bound}.
\end{proof}

\subsection{Examples} \label{examples}
By writing down the polynomial $Q$ of Theorem \ref{main}, one can find $n_{0}$ for which
$${\rm ord}_{\ell}(\kappa_{n}) = \mu_{\ell} \cdot \ell^{n} + \lambda_{\ell} \cdot n + \nu_{\ell}, $$
when $n \ge n_{0}$.  Once such an integer $n_{0}$ is known, one can find $\nu_{\ell}$ as well.  

All the multigraphs arising in abelian $\ell$-towers of bouquets are circulant multigraphs and thus are highly symmetrical.  Their vertices can be placed at the $\ell^{i}$-th roots of unity on the unit circle in the complex plane.  Their group $\mathbb{Z}/\ell^{i}\mathbb{Z}$ of symmetries becomes then apparent.

The computation of the number of spanning trees and of the polynomials $Q$ in this section have been performed with the software \cite{SAGE}.

\begin{enumerate}
\item If we take $a_{1} = a_{2} = 1$ and $\ell = 2$.  Then, we get:
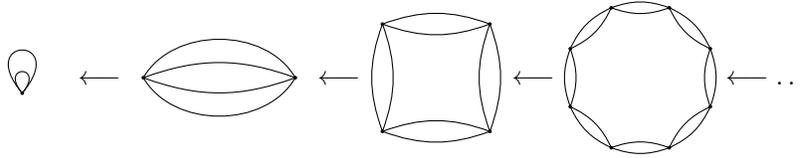
\begin{equation*}
\begin{tikzpicture}[baseline={([yshift=-1.7ex] current bounding box.center)}]
\node[draw=none,minimum size=2cm,regular polygon,regular polygon sides=1] (a) {};
\foreach \x in {1}
  \fill (a.corner \x) circle[radius=0.7pt];
\draw (a.corner 1) to [in=50,out=130,loop] (a.corner 1);
\draw (a.corner 1) to [in=50,out=130,distance = 0.5cm,loop] (a.corner 1);
\end{tikzpicture}
\longleftarrow \, \, \,
\begin{tikzpicture}[baseline={([yshift=-0.6ex] current bounding box.center)}]
\node[draw=none,minimum size=2cm,regular polygon,rotate = -45,regular polygon sides=4] (a) {};

  \fill (a.corner 1) circle[radius=0.7pt];
  \fill (a.corner 3) circle[radius=0.7pt];
  
  \path (a.corner 1) edge [bend left=20] (a.corner 3);
  \path (a.corner 1) edge [bend left=60] (a.corner 3);
  \path (a.corner 1) edge [bend right=20] (a.corner 3);
  \path (a.corner 1) edge [bend right=60] (a.corner 3);

\end{tikzpicture}
\, \, \, \longleftarrow 
\begin{tikzpicture}[baseline={([yshift=-0.6ex] current bounding box.center)}]
\node[draw=none,minimum size=2cm,regular polygon,regular polygon sides=4] (a) {};

\foreach \x in {1,2,...,4}
  \fill (a.corner \x) circle[radius=0.7pt];

\path (a.corner 1) edge [bend left=20] (a.corner 2);
\path (a.corner 1) edge [bend right=20] (a.corner 2);
\path (a.corner 2) edge [bend left=20] (a.corner 3);
\path (a.corner 2) edge [bend right=20] (a.corner 3);
\path (a.corner 3) edge [bend left=20] (a.corner 4);
\path (a.corner 3) edge [bend right=20] (a.corner 4);
\path (a.corner 4) edge [bend left=20] (a.corner 1);
\path (a.corner 4) edge [bend right=20] (a.corner 1);
\end{tikzpicture}
\longleftarrow
\begin{tikzpicture}[baseline={([yshift=-0.6ex] current bounding box.center)}]
\node[draw=none,minimum size=2cm,regular polygon,regular polygon sides=8] (a) {};

\foreach \x in {1,2,...,8}
  \fill (a.corner \x) circle[radius=0.7pt];
  
\foreach \y\z in {1/2,2/3,3/4,4/5,5/6,6/7,7/8,8/1}
  \path (a.corner \y) edge [bend left=20] (a.corner \z);
  
\foreach \y\z in {1/2,2/3,3/4,4/5,5/6,6/7,7/8,8/1}
  \path (a.corner \y) edge [bend right=20] (a.corner \z);
\end{tikzpicture}
\longleftarrow \ldots
\end{equation*}
The polynomial $Q$ is 
$$Q(T) = 2T,$$
so we should have $\mu_{2} = 1$ and $\lambda_{2} = 1$.  We calculate
$$\kappa_{0} = 1, \kappa_{1} = 2^{2}, \kappa_{2} = 2^{5}, \kappa_{3} = 2^{10}, \ldots, $$
but, it is simple to see combinatorially that
$$\kappa_{n} = 2^{2^{n} + n -1}.$$
We have
$${\rm ord}_{2}(\kappa_{n}) =  2^{n} + n -1,$$
for all $n \ge 0$.

\item If we take $a_{1} =3, a_{2} = 5$ and $\ell = 2$.  Then, we get:
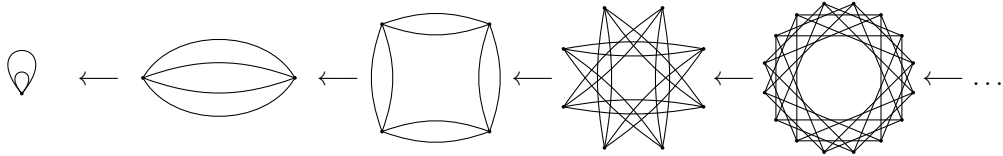
\begin{equation*}
\begin{tikzpicture}[baseline={([yshift=-1.7ex] current bounding box.center)}]
\node[draw=none,minimum size=2cm,regular polygon,regular polygon sides=1] (a) {};
\foreach \x in {1}
  \fill (a.corner \x) circle[radius=0.7pt];
\draw (a.corner 1) to [in=50,out=130,loop] (a.corner 1);
\draw (a.corner 1) to [in=50,out=130,distance = 0.5cm,loop] (a.corner 1);
\end{tikzpicture}
\longleftarrow \, \, \,
\begin{tikzpicture}[baseline={([yshift=-0.6ex] current bounding box.center)}]
\node[draw=none,minimum size=2cm,regular polygon,rotate = -45,regular polygon sides=4] (a) {};

  \fill (a.corner 1) circle[radius=0.7pt];
  \fill (a.corner 3) circle[radius=0.7pt];
  
  \path (a.corner 1) edge [bend left=20] (a.corner 3);
  \path (a.corner 1) edge [bend left=60] (a.corner 3);
  \path (a.corner 1) edge [bend right=20] (a.corner 3);
  \path (a.corner 1) edge [bend right=60] (a.corner 3);

\end{tikzpicture}
\, \, \, \longleftarrow 
\begin{tikzpicture}[baseline={([yshift=-0.6ex] current bounding box.center)}]
\node[draw=none,minimum size=2cm,regular polygon,regular polygon sides=4] (a) {};

\foreach \x in {1,2,...,4}
  \fill (a.corner \x) circle[radius=0.7pt];

\path (a.corner 1) edge [bend left=20] (a.corner 2);
\path (a.corner 1) edge [bend right=20] (a.corner 2);
\path (a.corner 2) edge [bend left=20] (a.corner 3);
\path (a.corner 2) edge [bend right=20] (a.corner 3);
\path (a.corner 3) edge [bend left=20] (a.corner 4);
\path (a.corner 3) edge [bend right=20] (a.corner 4);
\path (a.corner 4) edge [bend left=20] (a.corner 1);
\path (a.corner 4) edge [bend right=20] (a.corner 1);
\end{tikzpicture}
\longleftarrow
\begin{tikzpicture}[baseline={([yshift=-0.6ex] current bounding box.center)}]
\node[draw=none,minimum size=2cm,regular polygon,regular polygon sides=8] (a) {};

\foreach \x in {1,2,...,8}
  \fill (a.corner \x) circle[radius=0.7pt];
  
\foreach \y\z in {1/4,2/5,3/6,4/7,5/8,6/1,7/2,8/3}
  \path (a.corner \y) edge [bend left=10] (a.corner \z);
  
\foreach \y\z in {1/6,2/7,3/8,4/1,5/2,6/3,7/4,8/5}
  \path (a.corner \y) edge [bend left=10] (a.corner \z);
\end{tikzpicture}
\longleftarrow 
\begin{tikzpicture}[baseline={([yshift=-0.6ex] current bounding box.center)}]
\node[draw=none,minimum size=2cm,regular polygon,regular polygon sides=16] (a) {};

\foreach \x in {1,2,...,16}
  \fill (a.corner \x) circle[radius=0.7pt];
  
\foreach \y\z in {1/4,2/5,3/6,4/7,5/8,6/9,7/10,8/11,9/12,10/13,11/14,12/15,13/16,14/1,15/2,16/3}
  \path (a.corner \y) edge (a.corner \z);
  
\foreach \y\z in {1/6,2/7,3/8,4/9,5/10,6/11,7/12,8/13,9/14,10/15,11/16,12/1,13/2,14/3,15/4, 16/5}
  \path (a.corner \y) edge  (a.corner \z);
\end{tikzpicture}
\longleftarrow \ldots
\end{equation*}
The polynomial $Q$ is 
$$Q(T) = 34T - 56T^{2} + 36T^{3}-10T^{4} + T^{5},$$
so we should have $\mu_{2} = 0$ and $\lambda_{2} = 9$.  We calculate
$$\kappa_{0} = 1, \kappa_{1} = 2^{2}, \kappa_{2} = 2^{5}, \kappa_{3} = 2^{10}, \kappa_{4} = 2^{25}, \kappa_{5} = 2^{34} \cdot 577^{2}, \ldots $$
We have
$${\rm ord}_{2}(\kappa_{n}) =  9n -11,$$
for all $n \ge 4$.

\item  If we take $a_{1}=1, a_{2}=4, a_{3} =20$ and $\ell = 3$.  Then, we get:
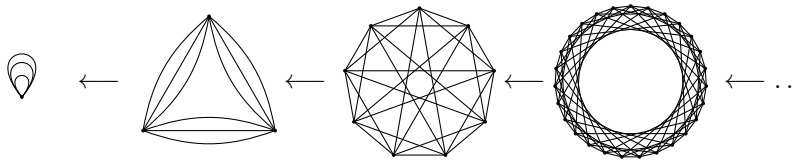
\begin{equation*}
\begin{tikzpicture}[baseline={([yshift=-1.7ex] current bounding box.center)}]
\node[draw=none,minimum size=3cm,regular polygon,regular polygon sides=1] (a) {};
\foreach \x in {1}
  \fill (a.corner \x) circle[radius=0.7pt];
\draw (a.corner 1) to [in=50,out=130,loop] (a.corner 1);
\draw (a.corner 1) to [in=50,out=130,distance = 0.8cm,loop] (a.corner 1);
\draw (a.corner 1) to [in=50,out=130,distance = 0.5cm,loop] (a.corner 1);
\end{tikzpicture}
\longleftarrow \, \, \,
\begin{tikzpicture}[baseline={([yshift=-0.6ex] current bounding box.center)}]
\node[draw=none,minimum size=2cm,regular polygon,regular polygon sides=3] (a) {};

\foreach \x in {1,2,3}
  \fill (a.corner \x) circle[radius=0.7pt];

\path (a.corner 1) edge [bend left=20] (a.corner 2);
\path (a.corner 1) edge [bend right=20] (a.corner 2);
\path (a.corner 2) edge [bend left=20] (a.corner 3);
\path (a.corner 2) edge [bend right=20] (a.corner 3);
\path (a.corner 3) edge [bend left=20] (a.corner 1);
\path (a.corner 3) edge [bend right=20] (a.corner 1);

\path (a.corner 1) edge  (a.corner 2);
\path (a.corner 2) edge  (a.corner 3);
\path (a.corner 3) edge  (a.corner 1);

\end{tikzpicture}
\longleftarrow \, \,
\begin{tikzpicture}[baseline={([yshift=-0.6ex] current bounding box.center)}]
\node[draw=none,minimum size=2cm,regular polygon,regular polygon sides=9] (a) {};

\foreach \x in {1,2,...,9}
  \fill (a.corner \x) circle[radius=0.7pt];
  
\foreach \y\z in {1/2,2/3,3/4,4/5,5/6,6/7,7/8,8/9,9/1}
  \path (a.corner \y) edge (a.corner \z);
  
\foreach \y\z in {1/3,2/4,3/5,4/6,5/7,6/8,7/9,8/1,9/2}
  \path (a.corner \y) edge (a.corner \z); 
  
\foreach \y\z in {1/5,2/6,3/7,4/8,5/9,6/1,7/2,8/3,9/4}
  \path (a.corner \y) edge (a.corner \z);

\end{tikzpicture}
\longleftarrow
\begin{tikzpicture}[baseline={([yshift=-0.6ex] current bounding box.center)}]
\node[draw=none,minimum size=2cm,regular polygon,regular polygon sides=27] (a) {};

\foreach \x in {1,2,...,27}
  \fill (a.corner \x) circle[radius=0.7pt];
  
\foreach \y\z in {1/2,2/3,3/4,4/5,5/6,6/7,7/8,8/9,9/10,10/11,11/12,12/13,13/14,14/15,15/16,16/17,17/18,18/19,19/20,20/21,21/22,22/23,23/24,24/25,25/26,26/27,27/1}
  \path (a.corner \y) edge (a.corner \z);
  
\foreach \y\z in {1/5,2/6,3/7,4/8,5/9,6/10,7/11,8/12,9/13,10/14,11/15,12/16,13/17,14/18,15/19,16/20,17/21,18/22,19/23,20/24,21/25,22/26,23/27,24/1,25/2,26/3,27/4}
  \path (a.corner \y) edge (a.corner \z); 
  
\foreach \y\z in {1/21,2/22,3/23,4/24,5/25,6/26,7/27,8/1,9/2,10/3,11/4,12/5,13/6,14/7,15/8,16/9,17/10,18/11,19/12,20/13,21/14,22/15,23/16,24/17,25/18,26/19,27/20}
  \path (a.corner \y) edge (a.corner \z);

\end{tikzpicture}
\, \, \longleftarrow \ldots
\end{equation*}
The polynomial $Q$ starts as follows 
$$Q(T) = 417T -13320T^{2}+175568T^{3} - \ldots - T^{20},$$
so we should have $\mu_{3} = 0$ and $\lambda_{3} = 5$.  We calculate
$$\kappa_{0} = 1, \kappa_{1} = 3^{3} , \kappa_{2} = 2^{6} \cdot 3^{8}, \kappa_{3} = 2^{6} \cdot 3^{13} \cdot 176417^{2}, \ldots $$
and we have
$${\rm ord}_{3}(\kappa_{n}) =  5n - 2,$$
for all $n \ge 1$.
\end{enumerate}

\bibliographystyle{plain}
\bibliography{main}

\end{document}